\documentclass{article}
\usepackage[margin=1.5in]{geometry}
\usepackage[utf8]{inputenc}
\usepackage[style=numeric]{biblatex}
\usepackage[dvipsnames]{xcolor}
\usepackage{color}
\usepackage{hyperref}
\hypersetup{
    colorlinks=true,
    linkcolor=SeaGreen,
    urlcolor=LimeGreen,
    citecolor=LimeGreen,
    linktoc=section}
\hypersetup{final=true}
 
\usepackage{textcomp}
\usepackage{float}
\usepackage{amsmath}
\usepackage{gensymb}
\usepackage{amsfonts}
\usepackage{amsthm}
\newtheorem{theorem}{Theorem}

\newtheorem{corollary}{Corollary}
\usepackage{url}
\usepackage{svg}
\usepackage{todonotes} 

\title{Hyperbolic Staircases: Periodic Paths on $2g+1$--gons}
\author{Mei Rose Connor, Diana Davis, Paige Helms, Michael Kielstra, \\
Samuel Leli\`evre,  Zachary Steinberg, and Chenyang Sun}

\begin{document}

\maketitle

\section{Introduction}

Suppose we are playing billiards on a table with straight edges. If we place a ball anywhere on the table and choose a direction to hit the ball in, the ball will travel in a straight line over time until it hits a wall and bounces off.  If we extend this infinitely far in both directions, the polyline traced across the table is called a \textit{billiard ball trajectory}.  If the ball returns to its starting point and direction after travelling a finite distance, we call the trajectory \textit{periodic}.  This leads naturally to the question of how to characterize the set of all periodic trajectories on a billiard table.

\begin{figure}[ht]
    \centering
    \includegraphics[width=0.4\textwidth]{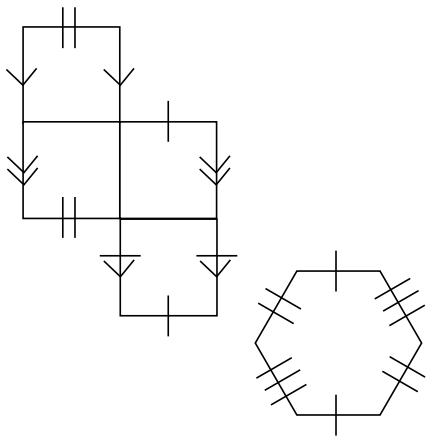}
    \caption{Two examples of translation surfaces. Edges with the same symbol are glued together. Note that the polygons in a translation surface need not be squares, have an even number of sides, or even regular polygons at all or other $2n$--gons.}
    \label{fig:sts}
\end{figure}

Rather than consider bent lines on two-dimensional Euclidean tables, we may ``unfold" the table and consider straight lines.  Every time we would reflect the trajectory at the edge of a table, we instead reflect the table across the edge and attach the reflected copy to the original.  Then, we can continue the trajectory into the reflected table along a straight line.  Furthermore, if we reflect a table twice in the same direction, we get back where we started, so we might as well have simply translated the position of our billiard ball by some constant amount.  This motivates the study of trajectories on \textit{translation surfaces} (see figure \ref{fig:sts}), which, following Wright, we define to be an equivalence class of polygons in the plane 
with parallel edges identified.  
For a classic example of a translation surface, consider the unit square $I^2$ in the plane with the left and right edges identified and the top and bottom edges identified; this surface is a torus and has a single polygonal region. We consider two collections of polygons to be equivalent if we can use cut-and-paste relations to go from one to the other; that is, we can begin with one collection of polygons and transition to the second collection by cutting individual polygons along straight lines and forming two new edges which are paired with each other, and then translating and re-gluing the identified edges together. 

As translation surfaces are embedded in the plane, we can apply any affine transformation of the plane to a translation surface. The collection of all 
area-preserving, affine transformations of the plane are called M\"obius transformations, and are of the form $$\left\{\begin{bmatrix}  a & b \\ c & d \end{bmatrix} | 
a,b,c,d \in \mathbb{R}, ad-bc=1\right\}.$$ In other words, translation surfaces carry an action of $SL(2, \mathbb{R})$.

This gives us a clear way to understand the cut-and-paste relations. Applying the matrix $\begin{bmatrix} 0 & -1 \\ 1 & 0 \end{bmatrix}$ to the square torus $I^2 \subset \mathbb{R}^2$, skews the square torus into a parallelogram. However, this parallellogram can be cut vertically into two triangles, which can be rearranged and glued along identified edges to recover the original square torus.
From Hubert \cite{Hubert}, we define the \textit{Veech group} of a translation surface to be the matrix subgroup of $SL(2, \mathbb{R})$ which preserves the surface up to cut-and-paste relations. 

A \textit{lattice} is a discrete, cocompact subgroup of $SL(2, \mathbb{R})$. If the Veech group of a translation surface forms a lattice in $SL(2, \mathbb{R})$ then we say that the translation surface is Veech. A \textit{cylinder} is a set of parallel trajectories (straight-line paths) on a translation surface which traverse the same sequence of edges, up to a cyclic permutation.  Veech surfaces have a number of pleasant properties, the most important of which involve cylinders. Most importantly for our purposes, any Veech surface decomposes into a union of cylinders, which implies that whether or not a trajectory on a Veech surface is periodic depends solely on the slope of the trajectory and not on the starting point. For example, on a square billiard table, every trajectory with the same slope bounces off the same walls, and this observation can be used to show the periodic trajectories on a square are exactly those with rational slopes.

The problem is also solved for all billiard tables shaped like regular polygons with more than four sides, although this is more difficult. In particular, the solution for the pentagon involves first deforming the regular pentagon into a rectilinear shape where slopes are easier to describe. All periodic slopes on this rectilinear shape are exactly of the form $a+b\phi$ with $a,b \in \mathbb{Q}$, where $\phi = \frac{1+\sqrt{5}}{2}$ is the golden ratio. However, two different papers have arrived at this answer about pentagonal billiard trajectories in two seemingly different ways.

The approach by Davis, Fuchs, and Tabachnikov \cite{DFT} relies on hyperbolic geometry. They represent slopes as points on the boundary of the Poincar\'e disk, and then show that all periodic slopes can be generated from a few simple slopes using various hyperbolic reflections. They show four possible ways to use a combination of hyperbolic reflection and rotation to create new points which also correspond to periodic slopes.

A different approach, by Davis and Lelièvre \cite{DL}, involves unfolding the pentagon into a translation surface. They create a series of successively more convenient translation surfaces - first a ``necklace", then a ``double pentagon", then a ``golden L" - and completely classifying all periodic trajectories on the golden L. From there we get four new ways of generating new periodic slopes given existing periodic slopes: four new matrices which take direction vectors to periodic direction vectors.

\begin{figure}[ht]
    \centering
    \includegraphics[width=0.7\textwidth]{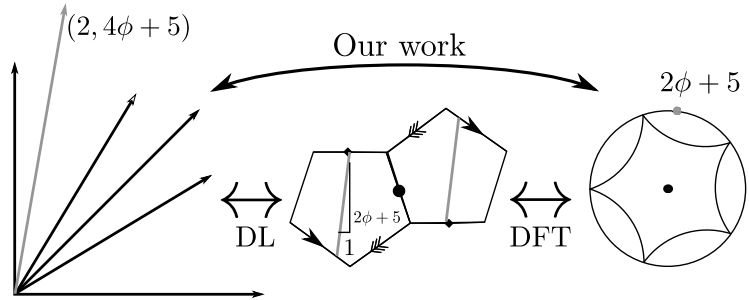}
    \caption{\cite{DFT} studies periodic trajectories using points on a Poincaré disk stereographically projected to a number line. \cite{DL} studies periodic trajectories using a first quadrant cut into four segments based on line segments derived from a translation surface. We prove that these two approaches are, in fact, isomorphic.}
    \label{fig:DFTandDLandOurWork}
\end{figure}

This paper will prove a natural equivalence between the two approaches. This allows the use of techniques from both papers to study paths on the pentagon, and offers a promising avenue for studying billiard trajectories on other regular polygons.

\subsection{The Davis-Fuchs-Tabachnikov Approach}

\begin{figure}[!ht]
    \centering
    \includegraphics[width=0.6\textwidth]{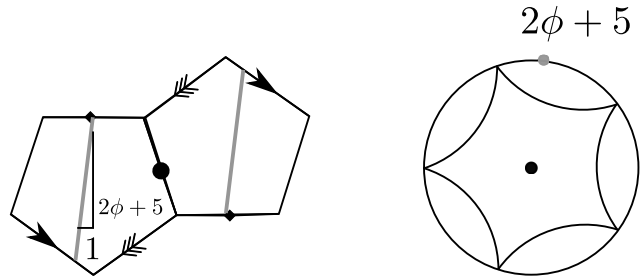}
    \caption{\cite{DFT} associates every possible slope of a periodic trajectory on a double pentagon, right, with a point on the boundary of the Poincar\'e disk, left. New periodic trajectories are generated by transforming Poincar\'e disk in a way that sends periodic slopes to new periodic slopes.}
    \label{fig:howDFTworks}
\end{figure}

Davis, Fuchs, and Tabachnikov classify billiard paths based on the slope of the initial direction in \cite{DFT}.  There is a canonical map between directions and points on a circle: each direction is mapped to the point where a ray in that direction starting at the center of the circle intersects the boundary.  The authors of \cite{DFT} therefore pass to considering points on a circle.

\begin{figure}[ht]
    \centering
    \includegraphics[]{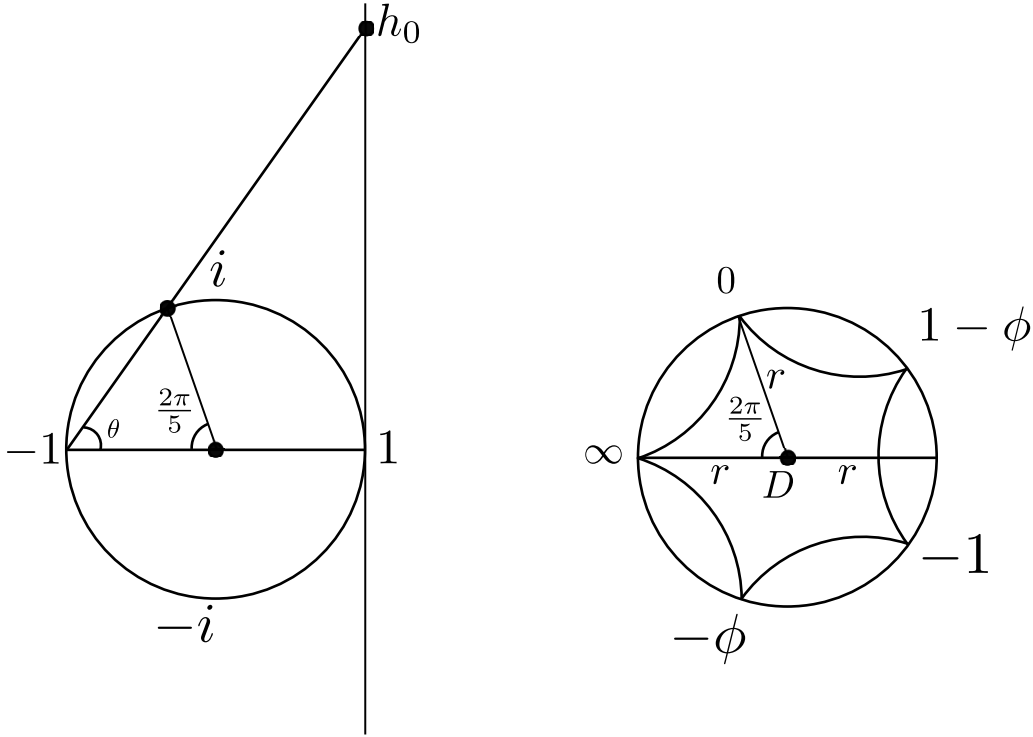}
    \caption{Left, a stereographic projection from the Poincar\'e disk to a line. The point $e^{i\frac{3\pi}{5}}$ on the boundary of the Poincar\'e disk is sent to the number $h_0$. Right, one particularly convenient stereographic projection used throughout this paper, which translates between hyperbolic land and staircase land, with points on the disk boundary labeled according to where they are sent under the projection. \\
    Points on the boundary of the Poincar\'e disk used in \cite{DFT} are mapped exactly to the slopes of their direction vectors used in \cite{DL}.  For example, the point $-1$ is sent to $\infty$. Of note is the fact that here we change the location of the zero point of the number line, so that it lines up with the projected location of one of the points of the pentagon.}
    \label{fig:proj}
\end{figure}

Their next task is to assign numerical values to specific points on that circle.  They work in $\mathbb{R} \cup \{\infty\}$, first choosing some point arbitrarily to be sent to $\infty$.  Then they set up a real number line tangent to the circle at the point directly opposite the point sent to $\infty$. A point $z$ on the circle recieves an associated real number by drawing a line between the point sent to $\infty$ and $z$ and seeing where it intersects the number line. This is called a \textit{stereographic projection}. More information can be found in \cite{Lee_2010}. 

To complete the setup, the authors consider the circle to be the boundary of a Poincaré disk model of the hyperbolic plane, where all the points to which the stereographic projection associates numbers are points at infinity of the hyperbolic plane. The authors claim that, by starting with a few of these points on the boundary of the Poincaré disk and applying various hyperbolic transformations, it is possible to generate all points corresponding to slopes of periodic trajectories on the \textit{double pentagon}, a Veech translation surface consisting of two pentagons with parallel sides identified.  From this, they can find the  slopes of periodic billiard paths on the pentagon.

Specifically, the point $\infty$ corresponds to a trajectory with infinite slope, that is, a vertical trajectory.  This is periodic on the double pentagon, so the authors start with it.  They then consider the Veech group of the double pentagon.  It is generated by two things: first, a rotation of $2\pi/5$, and second, a horizontal Dehn twist.  By considering what each of these do to points at infinity, they may be extended to, or viewed as, unique isometries of the Poincaré disk.  The rotation of $2\pi/5$ in the world of the pentagon becomes just that as a disk isometry, which immediately allows us to generate four more periodic slopes by rotating $\infty$ four times.  These five vertices form a hyperbolic pentagon.

Figure \ref{fig:proj} shows most of this, although, for ease of readability, only one edge of the hyperbolic pentagon is shown.  The point marked $0$ is projected to $h_0$ on the real line, and the angle $\theta$ is marked as the angle which must be calculated (by noting that triangle $0D\infty$ is isosceles) to determine where $h_0$ is.

Finally, Davis, Fuchs, and Tabachnikov use this hyperbolic construction to generate all periodic trajectories, as shown in Figure \ref{fig:hyperbolicmorepoints}. They calculate that the second generator of the Veech group, the Dehn twist, corresponds to a reflection in a particular edge of this hyperbolic pentagon.  Let $R$ be the reflection and $T$ be the rotation.  Then the authors consider the operators $RT^m$ for $m \in \{1, 2, 3, 4\}$.  (On the points at infinity, viewed as points on the real line via the stereographic projection, these work out to be linear fractional transformations.)  The effect of these operators on points in the privileged sector into which $R$ reflects all the others is particularly interesting: it consists of taking a point, rotating it into some other sector, and then reflecting it back into its original sector.  In this way we can build up a tree of directions, starting with a single point in the sector and repeatedly applying one of these four operators.

\begin{figure}[H]
    \centering
    \includegraphics[width=0.7\textwidth]{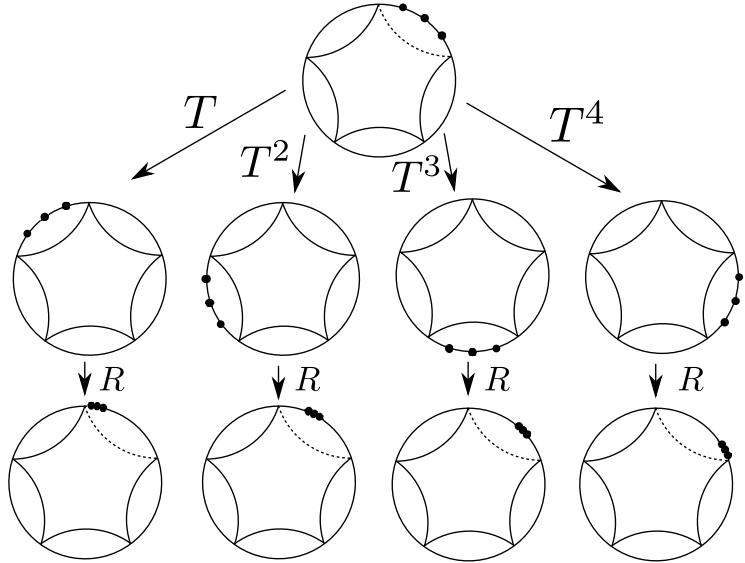}
    \caption{How \cite{DFT} generates new periodic points from old ones. Every possible slope of a trajectory, from $0$ to $\infty$, is assigned a point in the top-right arc. Each combination of $RT^i$, where $T$ is a rotation by $\frac{2\pi}{5}$ and $R$ is a hyperbolic reflection over the top right edge of the pentagon, sends points representing periodic slopes to other points representing periodic slopes.}
    \label{fig:hyperbolicmorepoints}
\end{figure}

Furthermore, since any periodic trajectory on the pentagon corresponds to a line in the tiling of the hyperbolic disk with pentagons, and since every application of $R$ to the vertices of a hyperbolic pentagon generates a new hyperbolic pentagon, we can get all periodic trajectory slopes by applying $RT^m$ to known ones, because by continually applying $RT^m$ we get all vertices of a tiling of the hyperbolic plane by pentagons and hence all directions of saddle connections on that tiling.

\subsection{The Davis-Lelièvre Approach}  
Davis and Leli\`evre \cite{DL} use flows on translation surfaces as a proxy for studying billiards on a pentagon. On a translation surface, represented by a polygon with edge identifications, trajectories that enter one edge emerge from the edge's identified pair in the same direction. Some examples of translation surfaces are a hexagon with opposite edges identified and the square--tiled surface seen in Figure \ref{fig:sts}. The final surface of interest can be represented by a rectilinear shape known as the ``Golden L", which is obtained by transforming the original pentagon into two intermediate surfaces: the ``necklace" and the double pentagon. 

To get the ``necklace", start with a regular pentagon and reflect it repeatedly to obtain 10 copies of the original pentagon that form a ring, as shown in Figure \ref{fig:doublepentagon}. This process is known as ``unfolding". Identifying parallel edges with the same label, the ``necklace" becomes a genus 6 surface with 5 vertices; a billiard trajectory on the original pentagon becomes a smooth trajectory on the new surface. We can then fold the necklace back in on itself to get a translation surface made out of two pentagons with opposite sides identified. Straight-line trajectories on the necklace become broken straight-line trajectories (that is, straight-line trajectories that cannot necessarily be drawn in the plane as a single straight line) on the double pentagon.

To further simplify the double pentagon, one may apply a matrix shear to alter some angles, and then use cut-and-paste operations to rearrange the resulting sheared double pentagon into an $L$-shaped polygon; the general process for any odd-sided polygon is discussed in section \ref{sec:staircase}. The resulting polygon is easier to study in many ways, being a union of rectangles. The reason such a transformation is permissible is that shearing both the surface and the trajectory, as well as cutting and pasting, does not alter the trajectory \textit{in relation to} the surface.\footnote{Note that this statement applies to flows on surfaces, but not billiard trajectories due to preservation of bounce angle and a lack of cut-and-paste equivalence for billiard tables. For a counterexample, take a direction parallel to an edge of a square: shearing both the square and the direction wildly alters the behavior of the resulting billiard trajectory.} More precisely stated as Corollary 2.4 in \cite{DL}: a direction on the double pentagon is periodic if and only if its sheared image is periodic on the Golden L.
    
\begin{figure}[ht]
    \centering
    \includegraphics[width=0.7\textwidth]{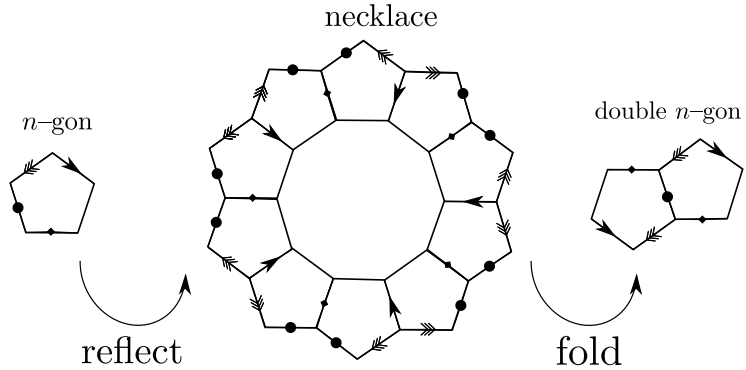}
    \caption{Construction of the ``double pentagon" translation surface from a single pentagon. The pentagon is first ``unfolded" to a ``necklace" through successive reflections, and then projected down to the double pentagon as a 5-fold cover. Under this transformation, billiard trajectories on the original pentagon, are not smooth at edges due to reflection, become smooth trajectories on the double surface. Edges with the same symbol and orientation are identified.}
    \label{fig:doublepentagon}
\end{figure}

Once the Golden L is obtained, one can immediately find five directions that yield periodic trajectories: orienting the L as one normally does, these directions are given by vectors starting at the lower-left vertex, directed at each of the other 5 vertices. Normalizing the central square to have unit side length, the vectors, with the same directions as those shown in figure \ref{fig:pentagonvectors}, are explicitly $\begin{bmatrix}1 \\ 0\end{bmatrix}$,$\begin{bmatrix}1 \\ \phi\end{bmatrix}$,$\begin{bmatrix}\phi \\ \phi\end{bmatrix}$,$\begin{bmatrix}\phi \\ 1\end{bmatrix}$, and $\begin{bmatrix}0\\1\end{bmatrix}$.

These five vectors naturally partition the first quadrant $\Sigma$ into four sectors. The sectors can be viewed as the image of $\Sigma$ under the matrices 
\[\sigma_0 = \begin{bmatrix}1&\phi\\0&1\end{bmatrix},
\sigma_1 = \begin{bmatrix}\phi&\phi\\1&\phi\end{bmatrix},
\sigma_2 = \begin{bmatrix}\phi&1\\\phi&\phi\end{bmatrix},
\sigma_3 = \begin{bmatrix}1&0\\\phi&1\end{bmatrix},\]
where the columns of the matrices are adjacent pairs of the above vectors.

\begin{figure}[ht]
    \centering
    \includegraphics[width=0.5\textwidth]{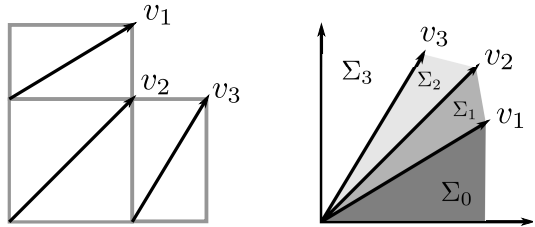}
    \caption{Vectors with the same directions as those used in \cite{DL} and how they are constructed from the golden L. Note that translating the vectors to start at zero visually sorts them by slope. The authors of \cite{DL} then construct matrices $\sigma_i$ which each send the region $\Sigma_i$ to the entire first quadrant. 
    In later sections, we generalize this approach to odd $n > 5$}
    \label{fig:pentagonvectors}
\end{figure}

All this construction culminates in a variety of deep and interesting results, one of which is Theorem 2.11 in \cite{DL}, also known as the Tree Theorem: every periodic direction on the Golden $L$ can be generated by taking $\begin{bmatrix}1\\0\end{bmatrix}$ and applying some finite sequence of matrices $\sigma_i$.

\section{Results}
\subsection{Summary}

\cite{DL} presents a set of four transformations $\{\Sigma_i\mid 0\le i\le 3\}$ which take periodic directions and generate new periodic directions. Similarly, \cite{DFT} presents a set of four transformations $\{RT^m\mid 0\le m\le 3\}$ which take points on the Poincar\'e disk representing periodic slopes to periodic slopes.

We demonstrate that these two methods are fundamentally the same. In addition, we produce a change of coordinate systems showing that the two papers' methods generate the same slopes of periodic trajectories in the same order via the same operations. We make this new coordinate system explicit and justify its use.

At the same time, we generalize the golden L construction from \cite{DL} to a ``staircase" construction for arbitrary $(2g+1)$-gons. By using results from the hyperbolic construction, via our new equivalence, we are able to calculate the slopes of the diagonals of the generalized staircase without carrying out complicated trigonometry on the rectangular surface.

In the process, we describe a few trigonometric identities, including a sequence which arises from a variant of the Chebyshev polynomials of the first kind.


\subsection{For the Pentagon}

We will first show how to turn the periodic points on the hyperbolic plane found by \cite{DFT} into the periodic trajectory vectors in $\mathbb{R}^2$ found by \cite{DL}. This takes the form of a bijective stereographic projection (the one shown in figure \ref{fig:proj}) from the hyperbolic geometry used in \cite{DFT} to slopes in $\mathbb{R}$.

We will now show how to construct this stereographic projection, which we will call $f$. Draw a pentagon inscribed the complex unit circle and consider two of its adjacent vertices $v_0$ and $v_1$.

A stereographic projection is fully specified by the data of three pairs of points and where they are sent. Like the construction in \cite{DFT}, our projection will project from the point $v_0$, meaning $f(v_0) = \infty$. However, instead of sending $v_1$ to a positive number, we translate the origin of the number line post-projection so that $f(v_1) = 0$. This means our projection wraps half of $\mathbb{R}$ into one-fifth of the border of the circle.

Finally, the point halfway on the arc between $v_0$ and $v_1$ is sent to $1$, which can be interpreted geometrically as scaling the circle's radius.

The transformation 
\[f(z) : \{ z \in \mathbb{C} \mid |z| = 1\} \to \mathbb{R} \cup \infty\] is explicitly defined as follows. First, fix the radius of the circle used in the stereographic projection,
\[r = \frac{1}{2} \left(\frac{\tan(\frac{\pi}{n})\tan(\frac{\pi}{2n})}{\tan(\frac{\pi}{n}) - \tan(\frac{\pi}{2n})}\right).\]

By placing a pentagon in the complex unit disk  with one vertex at $-1$, and stereographically projecting from $-1$, the transformation may then be expressed as
\[
f(z) = 2r\frac{\Im(z)}{1+\Re(z)}-\frac{1+\sqrt{5}}{4},
\]
where the $-(1+\sqrt{5})/4$ adjustment is done to ensure a certain vertex of the pentagon is sent to $0$. One can see by direct computation that $f(e^{\frac{3i \pi}{5}}) = 0$ and $f(-1) = \infty$.

\bigskip

This function will allow us to translate between the hyperbolic approach of \cite{DFT} and the staircase approach of \cite{DL}.

\begin{theorem}\label{thm:poincare}
    The points on the boundary of the Poincaré disk that are:
    \begin{enumerate}
        \item in the orbit of the vertices of the pentagon under the action of rotation by $2\pi/5$ and reflection in the edges of the pentagon, and
        \item mapped to positive numbers by our new stereographic projection
    \end{enumerate}
    are in fact mapped exactly to the slopes of trajectories on the golden L.  Furthermore, the group element which generates each point from our ``starting vertices" corresponds to the position of the trajectory associated to this slope in the tree of trajectories laid out in \cite{DL}. 
\end{theorem}
\begin{proof}
    From now on, we will dispense with talk of ``the point $z$ on the boundary of the hyperbolic plane mapped to $x$ by the stereographic projection $f$" and refer to the point $z$ by its label $x=f(z)$, directly identifying the disk's boundary with $\mathbb{R} \cup \{\infty\}$. Therefore, when we talk about ``the edge between 0 and $\infty$" we mean ``the arc on the hyperbolic plane between the points $z_0$ and $z_1$ where $f(z_1) = 0$ and $f(z_1) = \infty$".
    
    We are told in \cite{DFT}'s characterization of the pentagon's Veech group that the allowed operations on points on the edge of the disk are:
    \begin{enumerate}
        \item Reflection in one edge of the pentagon (which, in particular, we take to be the edge between $0$ and $\infty$), and
        \item Rotation of the entire Poincar\'e disk by $2\pi/5$.
    \end{enumerate}
    These operations are called $R$ and $T$ respectively, and, since both are isometries of the disk, both are linear fractional transformations.
    
    In the geometry which follows, we use the point labels from figure \ref{fig:proj}. 
    
    When constructing $f$ we chose two vertices of the pentagon and defined $f$ by sending them to specific values. We should first determine where $f$ sends the other three vertices of the pentagon.  To do this, we first find the position of $h_0$, the point such that $f(h_0)=0$, in figure \ref{fig:proj}. 
    
    We will now calculate $h_0$ using figure \ref{fig:proj} and some trigonometry. Let $D$ be the center of the hyperbolic disk. Since the triangle $D-0-\infty$ is isosceles, $\theta = \frac{1}{2}(\pi - 2\pi/5) = \frac{3\pi}{10}$, so the triangle with hypotenuse $\infty h_0$ has height $2r \tan (3\pi/10)$.
    
    Similarly, the point directly between $0$ and $\infty$ (which makes angle $2\pi/10$ with $D$ and $\infty$) should be sent to $2r \tan (2\pi/5)$.  But that point should also be sent to $1$.  This lets us calculate $r$, since we can set $$2r \tan (2\pi/5) - h_0 = 2r\left(\tan (2\pi/5) - \tan (3\pi/10)\right) = 1,$$ so $r = \frac{1}{4}\sqrt{\frac{1}{2}(5-\sqrt{5})}$.
    
    Armed with this, we can calculate the other three vertices: $$2r(\tan (i\pi/10) - \tan(3\pi/10))$$ for $i \in \{1, -1, -3\}$. This shows $f$ sends them to $1-\phi$, $-1$, and $-\phi$, in that order.
    
    Now recall that \cite{DFT} defined two operations on the hyperbolic plane, and hence on its boundary: $T(x)$, which rotates by $2\pi/5$, and $R(x)$, which reflects hyperbolically about the hyperbolic line between the points $0$ and $\infty$.  We can use our vertex values to write these as linear fractional transformations.  Specifically, we may take $T(x)$, rotation of the plane to be the unique LFT which sends each point to the previous: $$T(x) = \frac{-\phi x - 1}{x}.$$  Meanwhile, $R(x)$ exactly interchanges positive and negative points, and is its own inverse, so $R(x) = -x$.
    
    We define $S_m(x) = RT^m(x)$.  All points in the positive sector of the circle are of the form $S_aS_bS_c\dots(0)$.  (This is because each is generated by a string of $R$s and $T$s, and $RR$ is the identity.  We can start with only $0$ because we can get all our other ``starting vertices" by repeatedly using $T$ on $0$, and we have an $R$ as the furthest-left element because $T$ maps everything positive to something negative and $R$ maps everything negative to something positive, so they cancel each other out in this sense and ensure that the result is in the positive sector.)
    
    LFTs that operate on vector slopes can be directly associated with matrices that operate on vectors: the slope of $$\begin{pmatrix}a & b \\ c & d\end{pmatrix}\begin{bmatrix}x \\ y \end{bmatrix}$$ is exactly $$\frac{a+b(y/x)}{c+d(y/x)}.$$  To ensure that this map is a bijection, we insist that matrices associated with LFTs have determinant $1$.
    
    To calculate the matrix associated with $S_1$, we first note that $S_1 = \frac{1+\phi x}{x}$, so our matrix is $\begin{pmatrix}1 & \phi \\ 0 & 1 \end{pmatrix}$.  Similarly, $S_i = \sigma_{i - 1}$ for all $i$, where $\sigma_i$ is the matrix used in \cite{DL} to represent transformations into certain sectors of the golden L.
    
    In particular, we have demonstrated that $\sigma_{i-1}$ takes $\begin{bmatrix}1\\\alpha\end{bmatrix}$ to $S_i(\alpha)$.  Since scalar multiples commute with matrix multiplication, and since any vector of slope $\alpha$ is a scalar multiple of $\begin{bmatrix}1\\\alpha\end{bmatrix}$, we see that applying $\sigma_{i-1}$ to some vector $v$ gives a vector with the same slope as we would get by applying $S_i$ to the slope of $v$.  Therefore, for every product of $\sigma_{i-1}$s applied to $\begin{bmatrix}1\\0\end{bmatrix}$, the equivalent composition of $S_i$s applied to $0$ will generate the same slope.
\end{proof}

\subsection{Generalizing the Golden L}\label{sec:staircase}

As seen from \cite{DL}, many properties of the pentagon can be understood by studying a related translation surface, namely the Golden L. For regular $n$-gons with odd $n > 5$, we demonstrate how to construct a similar translation surface made of rectangles; a double heptagon ($n=7$) in Figure \ref{fig:doublengon} suffices for a visual example. For reasons to become clear later, we write $n=2g+1$ for some integer $g$.

\begin{figure}[ht]
    \centering
    \includegraphics[width=\textwidth]{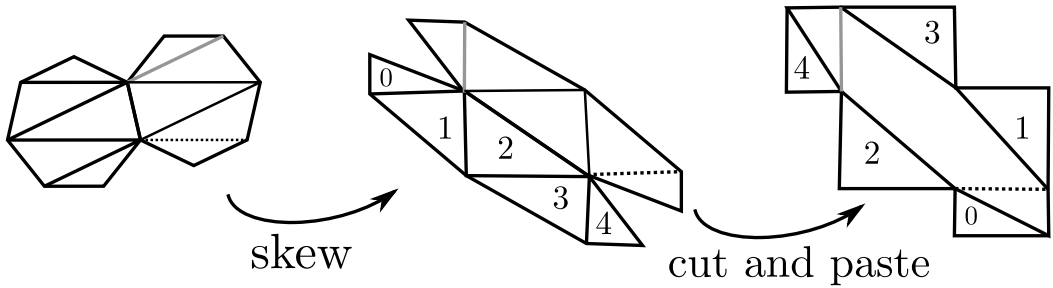}
    \caption{Construction of a staircase from a double $(2g+1)$-gon. Pairs of parallel sides are identified. For $g=2$ this staircase is the Golden L.}
    \label{fig:doublengon}
\end{figure}

We begin with a double $(2g+1)$-gon, and triangulate it in a zig-zag manner as shown, so that the two central triangles of the two polygons meet by the shared edge and form a rhombus. Note that this rhombus has two horizontal edges and two diagonal ones. Next, apply an affine transformation fixing the horizontal direction and taking the direction of the diagonal edges to vertical, so that the central rhombus becomes a rectangle. This transformation is known informally as a ``skew".

After the transformation, the original triangulation becomes one by right triangles. Observe the triangulation of the bottom-left polygon, and notice that by the central symmetry of the double polygon, every hypotenuse of the bottom-left triangulation corresponds to a hypotenuse of the upper-right triangulation. For each triangle in the bottom-left polygon, we translate it so that its hypotenuse matches its symmetric counterpart in the upper-right; note that the largest central triangle, in effect, is not moved. The resulting staircase-shaped translation surface (called the ``Golden L" for $g=2$ and ``Auramite W" for $g=3$) serves an analogous role in studying general double $(2g+1)$-gons as the Golden L in the double pentagon.

\subsection{Properties of the staircase surface}
\begin{figure}[ht]
    \centering
    \includegraphics[width=0.3\textwidth]{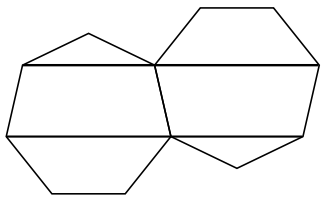}
    \caption{A cylinder decomposition of a double $(2g+1)$-gon..}
    \label{fig:heptagon-cylinders}
\end{figure}

\begin{theorem}
    A staircase surface constructed from a double $(2g+1)$-gon is Veech and decomposes into $g$ cylinders.
\end{theorem}
\begin{proof}
      The double $(2g+1)$-gon is Veech, so, after a uniform skew and a cut-and-paste operation, the result, the staircase surface, is also Veech.  Furthermore, such operations do not change the number of cylinders.  Consider a cylinder decomposition parallel to one side.  Then, as in figure \ref{fig:heptagon-cylinders}, we can pick one $(2g+1)$-gon in the double $(2g+1)$-gon and count the cylinders that go through it.  Each can be paired with two vertices of the $(2g+1)$-gon, the lowest two that it touches, leaving one vertex left over at the top.  Therefore the surface has $\frac{2g+1-1}{2}=g$ cylinders.
\end{proof}

\begin{figure}[ht]
    \centering
    \includegraphics[width=0.3\textwidth]{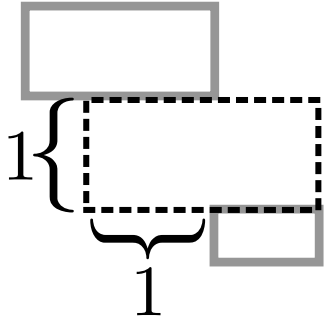}
    \caption{The cylinders for the horizontal flow along a staircase surface. Each rectangle has the same aspect ratio, given by Theorem \ref{theorem:staircasecylinders}}
    \label{fig:staircasecylinders}
\end{figure}

This kind of trigonometry generalizes to the general $n$-gon, for $n$ odd.

\begin{figure}[ht]
    \centering
    \includegraphics[width=0.5\textwidth]{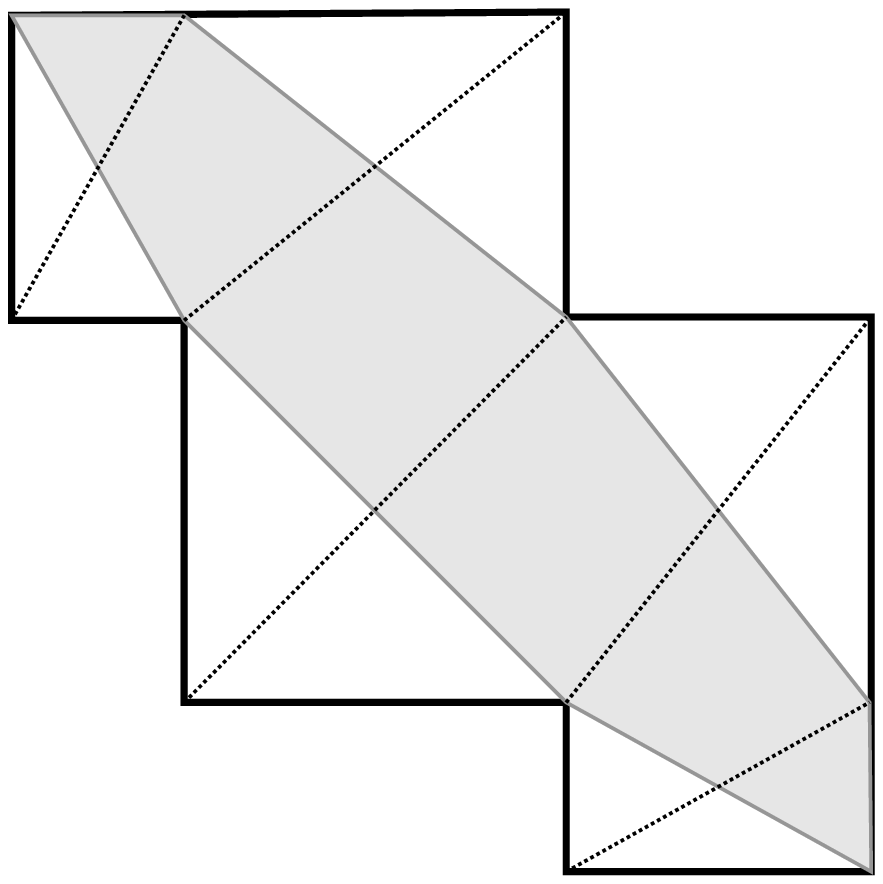}
    \caption{A staircase surface for $n=7$, with diagonals of each rectangle marked in dashed lines. The slopes of these dashed diagonals are given by Theorem \ref{theorem:staircasediagonals}}
    \label{fig:staircasediagonals}
\end{figure}

\begin{theorem}
\label{theorem:staircasediagonals}
    For $n$ odd, the slopes of the diagonals of the double $n$-gon staircase (the purple lines in Figure \ref{fig:staircasediagonals}) are $-v_{n-1}, -v_{n-2}, \dots, -v_{(n+1)/2}$, where $v_i = \sin \left(\frac{(1-i)\pi}{n}\right)/\sin\left(\frac{i\pi}{n}\right)$.
\end{theorem}
\begin{proof}
      We use an analogous stereographic projection and proof technique to that in the proof of theorem \ref{thm:poincare}, and call the values assigned to the vertices of the original hyperbolic $n$-gon $v_i$.  In this case, we fix the scale of $\ell$ and instead change the radius of the disk, so let that radius be $r$.

      We start by drawing an $n$-gon in the hyperbolic plane with vertices equally spaced around the boundary and one at $\infty$.  We designate one point on the circle, the vertex immediately counterclockwise of $\infty$, as $0$, and find the point $h_0$ to which it is projected.  The angle $0D\infty$, that is, the angle from the point $0$ to the center of the circle $D$ to the point $\infty$, is $\frac{2\pi}{n}$.  Since the triangle $0D\infty$ is isosceles, the other two angles are equal, so the angle $D\infty0$ is $\frac{1}{2}\left(\pi - \frac{2\pi}{n}\right) = \frac{\pi}{2} - \frac{\pi}{n}$.  Then the height of the triangle with hypotenuse $h_0\infty$ is $2r \tan \left(\frac{\pi}{2} - \frac{\pi}{n}\right)$.  However, $\tan (\pi/2 - x) = 1/\tan(x)$, so $h_0$ is at height $2r/\tan(\pi/n)$.
      
      We need to solve for $r$, so we choose one more value for one more point on the boundary.  We say that the point midway between $0$ and $\infty$ is $1$, motivated by the fact that the slope ``midway between" a horizontal and a vertical line is $1$.  We then use the assertion that $h_1 - h_0 = 1$ to set $$2r\left(\frac{1}{\tan (\pi/2n)} - \frac{1}{\tan(\pi/n)}\right) = 1.$$  Simplifying, this gives $$r = \frac{1}{2}\left(\frac{\tan (\pi/n) \tan (\pi/2n)}{\tan (\pi/n) - \tan (\pi/2n)}\right)$$.
      
      Now, the $i$-th vertex counterclockwise from $\infty$ will be at angle $\frac{2 \pi i}{n}$.  Performing the same calculation as for $h_0$ and $h_1$ one more time, we get $$v_i = 2r\left(\frac{1}{\tan (i\pi/n)} - \frac{1}{\tan (\pi/n)}\right).$$  By substituting in $r$, expanding out $\tan$ in terms of $\sin$ and $\cos$, and using the identity $\sin a \cos b - \cos a \sin b = \sin (a - b)$, this simplifies to $$v_i = \frac{\sin \frac{(1-i)\pi}{n}}{\sin\frac{i\pi}{n}}.$$
      
      The values $-v_i$, then, are the slopes of vectors generated by starting with $\infty$ and applying single transformations, each corresponding to one of the generators of the Veech group of the double $n$-gon: a rotation by $2\pi/n$ followed by a skew.   (This is described for $n = 5$ in \cite{Davis_2017}, but nothing in their argument is in any way special to this case and replacing $5$ with any odd $n$ in their argument gives precisely our claim.)  The values $v_i$ are the results of starting from $\infty$ and applying rotation, by construction, since they are spaced evenly around the boundary of the hyperbolic disk.
      
      The skew's action on the double $n$-gon staircase twists a given trajectory around a horizontal cylinder once, so it fixes the slopes $0$ and $\infty$, but no others.  According to \cite{Davis_2017}, it corresponds to an isometry of the hyperbolic disk, which means that its action on the boundary corresponds to a linear fractional transformation, and so it must be exactly the map $x \mapsto -x$.  Therefore, the values $-v_i$ are the result of applying first rotation and then skew.  But the slopes of the diagonals of the double $n$-gon staircase are exactly those slopes which can be generated from the vertical vector $\begin{bmatrix}0\\1\end{bmatrix}$ by applying rotations and skews.
      
      We restrict ourselves to $-v_{n-1}$ through $-v_{(n+1)/2}$ since the other $-v_i$s are the multiplicative inverses of these.
\end{proof}

\begin{corollary}\label{cylcor}
    Assuming a staircase surface has been uniformly stretched so that the central rectangle is square, as in 
    Figure \ref{fig:staircasecylinders}, each cylinder has aspect ratio $2\cos{\frac{\pi}{n}}$.
    \label{theorem:staircasecylinders}
\end{corollary}
\begin{proof}
    Since the staircase surface comes from a uniform skew of a regular polygon, all cylinders have the same aspect ratio.  This is, by inspecting the bottommost one, $$-v_{n-1} = \frac{\sin \frac{(n-2)\pi}{n}}{\sin \frac{(n-1)\pi}{n}} = \frac{\sin \frac{2\pi}{n}}{\sin \frac{\pi}{n}} = 2 \cos \frac{\pi}{n}.$$
\end{proof}

Some examples of the slopes we can calculate in this way, which we have also numerically verified by more direct computation, can be found in table \ref{table:slopes}.

\begin{table}[ht]
    \begin{tabular}{ll}
    $n$        & Slopes $\ge 1$                                                \\
    5          & 1, 1.6180                                           \\
    7          & 1, 1.2470, 1.8019                                  \\
    9          & 1, 1.1372, 1.3473, 1.8794                           \\
    11         & 1, 1.0882, 1.2036, 1.3979, 1.9190                 \\
    13         & 1, 1.0617, 1.1361, 1.2411, 1.4270, 1.9419        \\
    15         & 1, 1.0457, 1.0982, 1.1654, 1.2643, 1.4451, 1.9563
    \end{tabular}
    \caption{Slopes of diagonals of the generalized staircase figure for small $n$, calculated both directly and by using theorem \ref{theorem:staircasediagonals} and verified to match up to the first four decimal places. There are some additional slopes not included here, but those slopes are all $<1$ and of the form $1/s$ where $s$ is a slope on the chart.}
    \label{table:slopes}
\end{table}

\begin{theorem}
    For any $n$, there exist unique matrices $\sigma_1, \dots, \sigma_{n-1}$ such that:
    \begin{enumerate}
        \item $\sigma_i^{-1}$ takes $\begin{bmatrix}1 \\ -v_i\end{bmatrix}$ to a horizontal vector in the positive $x$-direction and $\begin{bmatrix}1 \\ -v_{i+1}\end{bmatrix}$ to a vertical one in the positive $y$-direction (where $\begin{bmatrix}1 \\ \infty \end{bmatrix}$ is taken to be $\begin{bmatrix}0 \\ 1 \end{bmatrix}$),
        \item the periodic trajectories on the generalized staircase figure have slopes $\sigma_{n_1}\sigma_{n_2}\cdots\sigma_{n_k}\begin{bmatrix}1\\0\end{bmatrix}$, and
        \item $\det \sigma_i = 1$.
    \end{enumerate}
\end{theorem}
\begin{proof}
      Consider the hyperbolic disk with the stereographic projection on its edge as in theorem \ref{theorem:staircasediagonals}.  Then the operations in this disk which generate all the points on the edge corresponding to periodic trajectory slopes on the $n$-gon are isometries of the projective plane (rotations and reflections), so, in terms of the stereographic projection on the boundary, they can be written as linear fractional transformations.  There are $n-1$ of them, each corresponding to a rotation of $\frac{2\pi i}{n}$ followed by a reflection.
      
      Any linear fractional transformation that operates on vector slopes can be identified with a determinant-$1$ matrix that operates on vectors themselves.  This gives us our $\sigma_i$s and we immediately get properties 2 and 3.  For property 1, consider the vectors $\begin{bmatrix}1 \\ -v_i \end{bmatrix}$.  The slopes are $-v_i$ and $-v_{i-1}$.  The operation $\sigma_i^{-1}$ will first perform a reflection to get the slopes $v_i$ and $v_{i-1}$, which are, respectively, $\frac{2\pi i}{n}$ and $\frac{2 \pi (i-1)}{n}$ radians away from the point $0$ by construction (as in the proof of theorem \ref{theorem:staircasediagonals}), and then rotate them $\frac{2\pi i}{n}$ radians to get the slopes $0$ and $\infty$ as required.
\end{proof}

\section{Future Work}
There are a number of potential avenues for future work in this area.  Our results for the slopes of diagonals of the double $n$-gon staircase, and our methods for proving them, bear similarities to the more abstract-algebraic approach of \cite{Calta_Smillie_2007}.  Further investigation is necessary to determine if these similarities are superficial or if they reveal deeper patterns.  At the same time, we note that the equivalence between the hyperbolic and staircase representations of periodic slopes goes both ways, but our results are restricted to facts about the staircase representation proved by way of the hyperbolic representation.  It may also be possible to derive results about the hyperbolic plane by means of known theorems centering around the staircase representation; whether this can be done is still an (admittedly vaguely-phrased) open problem.

\appendix
\section{Trigonometric Identities and Minimal Polynomials}

\begin{figure}[ht]
    \centering
    \includegraphics[width=0.7\linewidth]{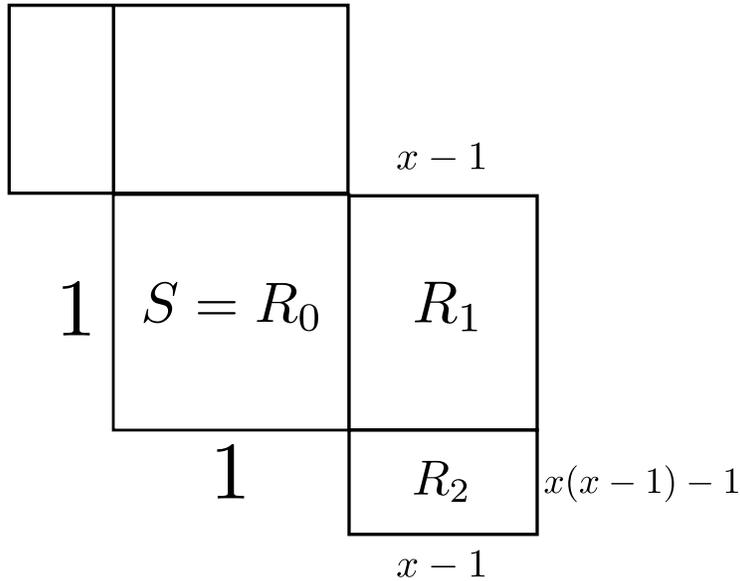}
    \caption{Staircase cylinders}
    \label{fig:scyl}
\end{figure}

We now have formulae for $v_i$ in terms of trigonometric functions, and can plug them into various facts that we can learn without using trigonometry at 
all.  These give new trigonometric identities.  First,we can expand out $$v_i = \frac{\tan (\pi/2n)}{\tan (\pi/n) - \tan (\pi/2n)}\left(\frac{\tan 
(\pi/n)}{\tan (i\pi/n)}-1\right).$$ 
Using this formula, we note that, by symmetry, $v_i = \frac{1}{v_{n+1-i}}$, and $\tan (\pi - x) = - \tan x$, so $$\frac{\tan (\pi/2n)}{\tan (\pi/n) - \tan
(\pi/2n)}\left(\frac{\tan (\pi/n)}{\tan (i\pi/n)}-1\right) = -\frac{\tan ((i-1)\pi/n)}{\tan (\pi/n) + \tan ((i-1)\pi/n)}\left(\frac{\tan (\pi/n)}{\tan 
(\pi/2n)}-1\right)$$ for $n$ odd.

Moreover, consider the aspect ratio of the cylinders in the double n-gon staircase. We start by fixing the central square $S$ to have side length $1$, and let $x>1$ be the aspect ratio of cylinders (long:short). Let the staircase be decomposed into rectangles as shown in figure \ref{fig:scyl}, and let $\ell(k),s(k)$ denote the length of the long and short sides of rectangle $R_k$. 

We first establish that $s(0)=1$, and $s(1)=x-1$ by considering a cylinder containing the central square. For $k\ge 2$, consider the cylinder formed by the rectangles $R_{k-1},R_{k}$, where the short side of $R_{k-1}$ meets the long side of $R_{k}$. The cylinder has short side $s(k-1)$, and long side $s(k)+s(k-2)$. By preservation of aspect ratio, we have $s(k)+s(k-2)=x(s(k-1))$, $s(k)=xs(k-1)-s(k-2)$. Let $P_k(x)$ denote the sequence of polynomials that satisfy the above recurrence with the same initial conditions, i.e., $P_0(x)=1, P_1(x)=x-1, P_k(x)=xP_{k-1}(x)-P_{k-2}(x)$.

For the double $(2m+1)$-gon staircase, the outermost rectangle is $R_{m-1}$, with short side $s(m-1)$ and long side $\ell(m-1)=s(m-2)$. Since the outermost rectangle is a complete cylinder, we have $\ell(m-1)=xs(m-1)$, which gives $0=xs(m-1)-s(m-2)=P_m(x)$. From corollary \ref{cylcor}, $P_{(n-1)/2}(-v_{n-1}) = P_{(n-1)/2}\left(2 \cos \frac{\pi}{n}\right)$.

Since the constant term of this is $1$, and since it is a polynomial of degree $(n-1)/2$ in $\cos \frac{\pi}{n}$, this gives the minimal polynomial of $\cos \frac{\pi}{n}$.  For instance, setting $n=5$ gives $$0=(-v_4)^2-(-v_4)-1 = 4 \cos^2 \frac{\pi}{5} - 2 \cos \frac{\pi}{5} - 1$$ while $n=7$ gives $$0 = (-v_6)^3 - (-v_6)^2 - 2(-v_6) + 1 = 8 \cos^3 \frac{\pi}{7} - 4 \cos^2 \frac{\pi}{7} - 4 \cos \frac{\pi}{7} + 1.$$

\printbibliography

\end{document}